\documentclass[a4paper,11pt]{article}
\usepackage[T1]{fontenc}
\usepackage[latin1]{inputenc}
\usepackage{amssymb}
\usepackage[english]{babel}
\usepackage{amsthm}
\usepackage{amsmath}
\usepackage{amsfonts}
\usepackage{enumerate}

\usepackage{array,lmodern,graphicx}

\def\Re{\text{Re\ }}

\def\R{\mathbb{R}}
\def\C{\mathbb{C}}

\def\fra{\mathfrak{a}}

\newtheorem{lemma}{Lemma}[section]
\newtheorem{corollary}[lemma]{Corollary}
\newtheorem{proposition}[lemma]{Proposition}
\newtheorem{theorem}[lemma]{Theorem}
\newtheorem{definition}[lemma]{Definition}
\newtheorem{remark}[lemma]{Remark}
\numberwithin{equation}{section}

\newcommand{\Hi}{\mathcal{H}}
\newcommand{\V}{\mathcal{V}}
\newcommand{\A}{\mathcal{A}}

\begin{document}

\title{ Lions' maximal regularity problem with  $H^{\frac{1}{2}}$-regularity in time}
\author{Mahdi Achache, El Maati Ouhabaz \thanks{\noindent Univ. Bordeaux, Institut de Math\'ematiques (IMB). CNRS UMR 5251. 351,  
Cours de la Lib\'eration 33405 Talence, France.
 Mahdi.Achache@math.u-bordeaux.fr, Elmaati.Ouhabaz@math.u-bordeaux.fr}}
\date{}

\maketitle
\begin{abstract}
We consider the problem of  maximal regularity  for non-autonomous Cauchy problems
 \begin{equation*}
 \left\{
  \begin{array}{rcl}
     u'(t) + A(t)\,u(t) &=& f(t), \ t \in (0, \tau] \\
     u(0)&=&u_0.
  \end{array}
\right.
\end{equation*}
The time dependent operators $A(t)$ are associated with  (time dependent) sesquilinear forms on a Hilbert space $\Hi$.  We are interested in J.L. Lions's problem concerning maximal regularity of such equations. We give a positive answer  to this problem under minimal regularity assumptions on the forms. Our main assumption is that the forms are  piecewise $H^{\frac{1}{2}}$  with respect to the variable $t$. This  regularity assumption is optimal and our results are the most general ones on this problem. 
\vspace{.7cm}

\noindent \textbf{keywords:} Maximal regularity, non-autonomous evolution equations, Sobolev regularity
\end{abstract}

\section{Introduction}\label{Sec1}
 Let 
 $( \mathcal{ H}, (\cdot, \cdot ),\| \cdot \|)$ be a separable Hilbert space over $\R$ or $\C$. We consider another separable  Hilbert space $\mathcal{V}$ which is densely and continuously embedded into $\mathcal{H}$. We denote by $\mathcal{V}'$ the (anti-) dual space of $\mathcal{V}$ so that 
$$\mathcal{ V} \hookrightarrow_{d} \mathcal{ H} \hookrightarrow_{d}  {\mathcal{V}}'.$$
Hence there exists a constant $C >0$ such that 
$$ \| u \| \leq C \| u \|_{\mathcal{V}} \     \ (u \in \mathcal{V}),$$
where $\| \cdot \|_{\mathcal{V}}$ denotes the norm of $\mathcal{V}$. Similarly, 
$$ \| \psi \|_{\mathcal{V}'}  \leq C \| \psi  \| \     \ ( \psi  \in \mathcal{H}).$$
We  denote by $\langle , \rangle$ the duality $\mathcal{V}'$-$\mathcal{V}$ and note that 
$\langle \psi, v \rangle = (\psi, v)$ if $\psi \in \mathcal{H}$ and $v \in \V$.  \\
We consider a family of sesquilinear forms $$   \fra :[0,\tau]\times \mathcal{V} \times  \mathcal{V}  \rightarrow  \C. $$
We assume throughout this paper the following usual properties. 
\begin{itemize}
 \item{} [H1]: $D(\fra(t))= \mathcal{V}$ (constant form domain),
 \item{}[H2]: $ | \fra(t,u,v)|\leq M \| u \|_\mathcal{ V} \| v \|_\mathcal{ V}$  for $t \in [0, \tau], u, v \in \V$ and some constant $M > 0$ (uniform boundedness),
\item{}[H3]: $\Re \fra(t,u,u)+\nu \| u \|^2\geq \delta  \| u \|_\mathcal{ V}^2 $ for $u \in {\mathcal V}$ and  some $\delta > 0$ \text{and  } $\nu \in  \R $ (uniform quasi-coercivity).
\end{itemize}
We denote by $A(t),\mathcal{ A}(t)$ the usual associated operators with  $\fra(t)$ as operators on $\mathcal{H}$ and ${\mathcal{V}}'$, respectively. In particular, $\mathcal{A}(t) : \mathcal{V} \to \mathcal{V}'$ as a bounded operator and 
$$\fra(t,u,v) = \langle \mathcal{A}(t) u, v \rangle \ {\rm for\ all\ } u, v \in \mathcal{V}.$$
The operator $A(t)$ is the part of $\mathcal{A}(t)$ on $\mathcal{H}$. 

We consider the non-homogeneous Cauchy problem
\begin{equation}\label{eq:evol-eq} \tag{P}
\left\{
  \begin{array}{rcl}
     u'(t) + A(t)\,u(t) &=& f(t), \ t \in (0, \tau] \\
     u(0)&=&u_0. 
  \end{array}
\right.
\end{equation}
\begin{definition}\label{def1.1}
The Cauchy problem (P) has maximal $L^2$-regularity in $\Hi$ if for every $f \in L^2(0, \tau; \Hi)$, there exists a unique 
$u \in H^1(0, \tau; \Hi)$ with $u(t) \in D(A(t))$ for a.e. $t \in [0, \tau]$ and $u$ is a solution of (P) in the $L^2$-sense.
\end{definition}

By a very well known result of J.L. Lions, maximal $L^2$-regularity always holds in the space $\V'$. That is,  for every $f \in L^{2}(0,\tau;\V')$ and $u_{0} \in \Hi$ there exists a unique $u \in H^{1}(0,\tau; \V')\cap L^{2}(0,\tau; \V) $ which  solves the equation 

\begin{equation}\label{eq:evol-eq} \tag{P'}
\left\{
  \begin{array}{rcl}
     u'(t) + \mathcal{A}(t)\,u(t) &=& f(t), \ t \in (0, \tau] \\
     u(0)&=&u_0.
  \end{array}
\right.
\end{equation}

In applications one needs however  maximal regularity in $\Hi$ (for example for elliptic boundary value problems one has to work on $\Hi$ rather than $\V'$ in order to identify the boundary conditions). Maximal regularity in $\Hi$ differs considerably from the same property in $\V'$. Before we recall known results and explain our main contribution in this paper, we recall that one of the reasons why maximal regularity (both in the autonomous and non-autonomous cases) was  intensively studied is due to the fact that it is a very useful tool to prove existence results for non-linear evolution equations. 

For symmetric forms $\fra(t)$, Lions \cite{Lions:book-PDE} (IV Sec. 6, Th\'eor\`eme 6.1]) proved  that if $t \mapsto \fra(t,u,v) $ is $C^1$ and $u_0 = 0$, then maximal $L^2$-regularity in $\Hi$ is satisfied. For general $u_0 \in D(A(0))$, Lions imposes the stronger regularity property  that $t \mapsto \fra(t,u,v) $ is $C^2$. 
Bardos \cite{Bar} improves the latter result for forms satisfying the uniform Kato square root property (see Definition \ref{def2.4} below) by assuming 
that $A(.)^{\frac{1}{2}}$ 
is continuously differentiable with values in $\mathcal{L}(\V, \V')$ and $u_0 \in \V$. Ouhabaz and Spina \cite{OS} proved maximal regularity in $\Hi$ if
$t \mapsto \fra(t,u,v) $ is $C^\alpha$ for some $\alpha > \frac{1}{2}$ when  $u_0 = 0$. This result was extended in Haak and Ouhabaz \cite{HO15} who prove maximal $L^p$-regularity under a slightly  better regularity condition and allowing  $u_0 \in D(A(0)^{\frac{1}{2}})$.  Dier \cite{Di} proved maximal 
$L^2$-regularity for symmetric forms such that $t \mapsto \fra(t,u,v)$ is of bounded variations.  Fackler \cite{Fac} proved that the order $\alpha > \frac{1}{2}$
in \cite{OS} or \cite{HO15}  is optimal in the sense that there exist $\fra(.)$ symmetric and $C^{\frac{1}{2}}$ for which maximal regularity in $\Hi$ fails. A counter-example already appeared in Dier \cite{Di} and it is based on a form which does not satisfy the Kato square root property. 
Dier and Zacher \cite{DZ} proved that if  $t \mapsto  \mathcal{A}(t)$ is in the fractional Sobolev space $H^{\frac{1}{2} + \delta}(0, \tau; \mathcal{L}(\V, \V'))$ for some $\delta > 0$ then maximal $L^2$-regularity in $\Hi$  holds. For a Banach space version of this result, see Fackler \cite{Fac2}. \\
The example in \cite{Fac} is not a differential operator. For elliptic operators in divergence form on $\R^n$, Auscher and Egert \cite{AE} proved maximal regularity if the coefficients satisfy a certain BMO-$H^{\frac{1}{2}}$ condition. The example from \cite{Fac} also shows that $ \mathcal{A}(.) \in W^{\frac{1}{2}, p}(0, \tau; \mathcal{L}(\V, \V'))$ for $p > 2$ is not enough to obtain maximal regularity. The example in \cite{Di} shows that  $\mathcal{A}(.) \in W^{\frac{1}{2}, p}(0, \tau; \mathcal{L}(\V, \V'))$ for $p < 2$ does not imply maximal regularity, at least for form which does not satisfy Kato's square root property. For a discussion on these negative results, see the review paper of Arendt, Dier and Fackler \cite{ADF}. As pointed in \cite{ADF}, the remaining problem is the case of fractional  regularity $H^{\frac{1}{2}}$. We solve this  problem  in the present  paper. Our main result shows that for forms satisfying the uniform Kato square root property and an integrability condition (see \eqref{eqHyp} below), if $t \mapsto  \mathcal{A}(t)$ 
is piecewise in the Sobolev space $H^{\frac{1}{2}}(0, \tau; \mathcal{L}(\V, \V'))$ then maximal $L^2$-regularity in $\Hi$ is satisfied. The initial data $u_0$ is arbitrary in $\V$. This result is optimal. The required Soblev regularity cannot be smaller than $\frac{1}{2}$ since $C^{\frac{1}{2}} \subset H^\alpha$ for 
$\alpha < \frac{1}{2}$. In the case where $\mathcal{A}(t) - \mathcal{A}(s)$ maps into the dual  space of  $[H, \V]_\gamma$ we allow the fractional Sobolev regularity to be $\frac{\gamma}{2}$. This  extends related results in Ouhabaz \cite{Ou15} and Arendt and Monniaux \cite{AM}. 

We give  the precise statements of the main results in the next section. In Sections \ref{sec2} and \ref{sec3} we prove several key estimates and develop the necessary  tools for the proofs of the main results. Some of these tools are quadratic estimates and $L^\infty(0, \tau; \V)$-estimates for the solution of the Cauchy problem. The main results are proved in Section \ref{sec4} and several examples are given in Section \ref{sec5}.

\section{Main results}\label{sec-main}

In this section we state explicitly our main results.  
For clarity of exposition we  consider separately the cases $\gamma = 1$, $\gamma \in (0, 1)$ and $\gamma = 0$.  

We start by  recalling  the definition of vector-valued fractional Sobolev spaces. 

\begin{definition}\label{defSob}
Let X be a Banach space, $\alpha \in (0,1)$ and $I$ an open subset of $\R$.   A function $f \in L^2(I; X)$ is  in the  fractional Sobolev space 
$H^{\alpha}(I;X)$  if
$$\|f\|^{2}_{H^{\alpha}(I;X)} := \|f\|^{2}_{L^{2}(I;X)} + \int_{I \times I} \frac{\|f(t)-f(s)\|_{X}^{2}}{|t-s|^{2\alpha+1}} \, ds dt < \infty.$$
We say that $f$ is in the homogeneous Sobolev space $\dot{H}^\alpha(I;X)$ if 
$$\|f\|^{2}_{\dot{H}^{\alpha}(I;X)} := \int_{I \times I} \frac{\|f(t)-f(s)\|_{X}^{2}}{|t-s|^{2\alpha+1}} \, ds dt < \infty.$$
\end{definition}

We shall say that  $f$ is piecewise in $H^{\alpha}(I;X)$ (resp. $\dot{H}^{\alpha}(I;X)$) if there exists $t_0 < t_1 <...< t_n $ such that $I = \cup_i [t_i, t_{i+1}]$ and  the restriction of 
$f$ to each sub-interval $(t_i, t_{i+1})$ is in $H^{\alpha}(t_i, t_{i+1} ; X)$ (resp. $\dot{H}^{\alpha}(t_i, t_{i+1} ; X)$). 

\bigskip
Let $\fra(t) : \V \times \V \to \C$ for $ 0 \le t \le \tau$ be a family of forms satisfying [H1]-[H3] and let $A(t)$ and $\mathcal{A}(t)$ be the associated operators on $\Hi$ and $\V'$, respectively.  We shall need the following property.
\begin{align}
& {\rm Given}\   \varepsilon > 0,\  {\rm there\ exists}\   \tau_0 = 0 < \tau_1 < ... < \tau_n = \tau \ {\rm  such\ that} \nonumber \\
&\sup_{t \in (\tau_{i-1}, \tau_i)}  \int_{\tau_{i-1}}^{\tau_i}   \frac{\| \mathcal{A}(t) - \mathcal{A}(s) \|^2_{\mathcal{L}(\V, \V')}}{ |t-s|} \, ds < \varepsilon. \label{eqHyp}
\end{align}

Note that this assumption is satisfied in many cases. Suppose for example that  $t \mapsto \fra(t,u, v)$ is $C^\alpha$ for some $\alpha > 0$ in the sense that 
\begin{equation}\label{C-alpha}
| \fra(t,u,v) - \fra(s,u,v) | \le M |t-s|^\alpha \|u \|_\V \| v \|_\V 
\end{equation} 
for some positive constant $M$ and all $u, v \in \V$. Then clearly 
\[
\| \mathcal{A}(t) - \mathcal{A}(s) \|_{\mathcal{L}(\V, \V')} \le M |t-s|^\alpha
\]
and this implies \eqref{eqHyp}. More generally, if $\omega_i$ denotes the modulus of continuity of $\mathcal{A}$ on the interval
$(\tau_{i-1}, \tau_i)$ then \eqref{eqHyp} is satisfied if 
\begin{equation}\label{eqmod}
\int_{ |r| \le \tau_i - \tau_{i-1} } \frac{\omega_i(r)^2}{r} \, dr < \varepsilon.
\end{equation}

\begin{theorem}\label{thm1}
Suppose that [H1]-[H3] and the  uniform Kato square  property are satisfied. If $t \mapsto \mathcal{A}(t)$ is piecewise in  $H^{\frac{1}{2}}(0, \tau; \mathcal{L}(\V, \V'))$ and satisfies \eqref{eqHyp}  then (P) has maximal $L^2$-regularity in $\Hi$ for all $u_0 \in \V$.  In addition, there exists a positive constant $C$ independent of $u_0$ and $f$ such that 
\begin{equation}\label{eq-apriori}
\| A(.) u(.) \|_{L^2(0, \tau; \Hi)} + \| u \|_{H^1(0, \tau ; \Hi)} \le C \left[ \| u_0 \|_{\V} + \| f \|_{L^2(0, \tau; \Hi)} \right].
\end{equation}
\end{theorem} 

We refer to the next  section for the definition of the uniform Kato square root property and for few more details on such property.

We have the following corollary which recovers the maximal regularity result proved  \cite{DZ} under the assumption that $t \mapsto \mathcal{A}(t)$ is  in  $H^{\frac{1}{2} + \delta}(0, \tau; \mathcal{L}(\V, \V'))$. 
\begin{corollary}\label{corDZ}
 Suppose that [H1]-[H3] and the  uniform Kato square  property are satisfied. Suppose that  $t \mapsto \mathcal{A}(t)$ is piecewise in  $H^{\frac{1}{2} + \delta}(0, \tau; \mathcal{L}(\V, \V'))$ for some $\delta > 0$.   Then (P) has maximal $L^2$-regularity in $\Hi$ for all $u_0 \in \V$.  In addition, there exists a positive constant $C$ independent of $u_0$ and $f$ such that 
\begin{equation}\label{eq-apriori}
\| A(.) u(.) \|_{L^2(0, \tau; \Hi)} + \| u \|_{H^1(0, \tau ; \Hi)} \le C \left[ \| u_0 \|_{\V} + \| f \|_{L^2(0, \tau; \Hi)} \right].
\end{equation}
\end{corollary}
\begin{proof} It follows from \cite{Sim}, Corollary 26 that $H^{\frac{1}{2} + \delta}(\tau_{i-1}, \tau_i; \mathcal{L}(\V, \V'))$ is continuously embedded into $C^{\delta}( \tau_{i-1}, \tau_i; \mathcal{L}(\V, \V')) $. As explained above this implies condition \eqref{eqHyp}. We then apply Theorem \ref{thm1}. \end{proof} 

Let $\gamma \in [0, 1]$ and  $\V_{\gamma} := [\Hi,\V]_{\gamma}$ be the usual complex interpolation space between $\Hi$ and $\V$. We denote by $\V_\gamma' := (\V_\gamma)'$ its (anti-) dual.  In some situations, one may have  $\mathcal{A}(t) -\mathcal{A}(s) : \V \to \V_\gamma'$ for some $\gamma \in [0, 1)$ (see Section \ref{sec2} for some additional details).  For example, this happens for forms $\fra(t)$ associated with differential operators such that  the difference  $\fra(t)- \fra(s)$ has only terms of smaller order or boundary terms.   In this case the required regularity in the previous theorem can be improved. Before we state the results we introduce the following assumption
\begin{itemize}
 \item{} [H4]:  $\| \mathcal{A}(t) - \mathcal{A}(s) \|_{\mathcal{L}(\V, \V_\gamma')} \le M_\gamma$
 \end{itemize}
 for some positive constant $M_\gamma$ and all $t, s \in [0, \tau]$. 

\begin{theorem}\label{thm2}
Suppose [H1]-[H3] and  that $D(A(t_0)^{\frac{1}{2}}) = \V$ for some $t_0 \in [0, \tau]$. Let $\gamma \in (0, 1)$ and suppose [H4].  If  $t \mapsto \mathcal{A}(t)$ is   piecewise in the homogeneous Sobolev space $\dot{H}^{\frac{\gamma}{2}}(0, \tau; \mathcal{L}(\V, \V_\gamma'))$,  then (P) has maximal $L^2$-regularity in $\Hi$ for all $u_0 \in \V$. In addition, there exists a positive constant $C$ independent of $u_0$ and $f$ such that 
\begin{equation}\label{eq-apriori-1}
\| A(.) u(.) \|_{L^2(0, \tau; \Hi)} + \| u \|_{H^1(0, \tau ; \Hi)} \le C \left[ \| u_0 \|_{\V} + \| f \|_{L^2(0, \tau; \Hi)} \right].
\end{equation}
\end{theorem}

Note that if [H4] holds with $\gamma = 0$ then
\[
| \langle \mathcal{A}(t)u - \mathcal{A}(s) u, v \rangle | \le M_0 \| u \|_{\V} \| v \|
\]
for all $u, v \in \V$.  In that case $\mathcal{A}(t) - \mathcal{A}(s)$ defines  a bounded operator from $\V$ into  $\Hi$. This implies in particular that the operators $A(t)$ have the same domain 
$D(A(t)) = D(A(0))$. For operators satisfying the later property  maximal regularity of (P) holds under the  assumption  that $t \mapsto A(t)v$ is relatively continuous for all $v \in D(A(0))$, see \cite{ACFP}, Theorem 3.3. See also \cite{GV} where the later regularity assumption  is replaced  by a certain commutation property. We prove here that maximal regularity holds without requiring any property on  the operators (or the forms). More precisely, we have
\begin{proposition}\label{prop3}
Suppose [H1]-[H3]. Suppose that [H4] holds for $\gamma = 0$ and that $D(A(t_0)^{\frac{1}{2}}) = \V$ for some $t_0 \in [0, \tau]$.  Then (P) has maximal $L^2$-regularity in $\Hi$ for all $u_0 \in \V$. In addition, there exists a positive constant $C$ independent of $u_0$ and $f$ such that 
\begin{equation}\label{eq-apriori-10}
\| A(.) u(.) \|_{L^2(0, \tau; \Hi)} + \| u \|_{H^1(0, \tau ; \Hi)} \le C \left[ \| u_0 \|_{\V} + \| f \|_{L^2(0, \tau; \Hi)} \right].
\end{equation}
\end{proposition}


\section{Preparatory lemmas}\label{sec2}

In this section we prove  several estimates  which will play an important  role in the proofs of the main results. We emphasize that one of the important points here  is to prove estimates with constants which are independent of $t$.  

Before we start let us point out that we may assume without loss of generality that assumption [H3] is satisfied with $\nu = 0$, that is the  forms are coercive with constant $\delta > 0$ independent of $t$. The reason is that the maximal regularity  of (P) is equivalent to the same property for  
\begin{equation}\label{CPnu}
v'(t) + (A(t) + \nu) v(t) = g(t), \    v(0) = u_0.
\end{equation}
This can be seen by observing that for $g(t) := f(t) e^{-\nu t}$, then $v(t) = u(t) e^{-\nu t}$ and clearly $v \in H^1(0, \tau; \Hi)$ {\it if and only if } $u \in H^1(0, \tau; \Hi)$ (and obviously
$f \in L^2(0, \tau; \Hi)$ {\it if and only if} $g \in L^2(0, \tau; \Hi)$). \\
When [H3] holds with $\nu = 0$ then clearly the operators $A(t)$ are invertible on $\Hi$. In addition, one has the resolvent estimate
\begin{equation}\label{eq2.0}
\| (\mu + A(t) )^{-1} \|_{\mathcal{L}(\Hi)} \le \frac{C}{ 1 + \mu}
\end{equation}
for all $\mu \ge 0$. The constant is independent of $t \in [0, \tau]$ (see e.g., \cite{ADLO}, Proposition 2.1). 
The same estimate  holds for $\mathcal{A}(t)$ on $\V'$. 

Recall that  $\V_{\gamma}=[\Hi,\V]_{\gamma}$ (for $\gamma \in [0,1]$)  is the complex interpolation space between $\Hi$ and $\V$ and $\V_\gamma' := (\V_\gamma)'$ denotes  its (anti-) dual space. 
\begin{lemma}\label{lem2.1} Suppose that [H1]-[H3] are satisfied with $\nu = 0$. Then there exists a constant $C > 0$ such that the following estimates hold  for all $\mu \ge 0$, $r > 0$ and all $t \in [0, \tau]$.  
\begin{itemize}
    \item[1-] $\|(\mu+ \A(t))^{-1}\|_{\mathcal{L}(\V_{\gamma}',\V)} \le  \frac{C}{(\mu+1)^{\frac{1-\gamma}{2}}},$
    \item[2-] $\|(\mu+ \A(t))^{-1}\|_{\mathcal{L}(\V_{\gamma}',\Hi)} \le  \frac{C}{(\mu+1)^{1-\frac{\gamma}{2}}},$
    \item[3-]$\|e^{- rA(t)}\|_{\mathcal{L}(\V_{\gamma}', \Hi)} \le \frac{C}{ r^{\frac{\gamma}{2}}}.$
\end{itemize}
\end{lemma}
\begin{proof} The arguments are classical but we write them here for clarity of the exposition. \\
We have for $w \in \V'$
\begin{eqnarray*}
\delta \| (\mu + \A(t))^{-1} w \|_{\V}^2 &\le& \Re \langle \A(t) (\mu + \A(t))^{-1} w, (\mu + \A(t))^{-1} w \rangle\\
&=& \Re \langle w- \mu (\mu + \A(t))^{-1} w, (\mu + \A(t))^{-1} w \rangle\\
&\le& \| w \|_{\V'} \| (\mu + \A(t))^{-1} w \|_{\V} +  C \| (\mu + \A(t))^{-1} w \|_{\V},
\end{eqnarray*} 
which gives $\|(\mu+ \mathcal{A}(t))^{-1}\|_{\mathcal{L}(\V',\V)} \le C'$. A similar argument gives the estimate $\|(\mu+\A(t))^{-1}\|_{\mathcal{L}(\Hi,\V)} \le \frac{C}{(\mu+1)^{\frac{1}{2}}}$.  The first assertion  follows by interpolation. 
For the second one we use 
\begin{align*}
&\|(\mu+ \A(t))^{-1}\|_{\mathcal{L}(\V_{\gamma}',\Hi)} \\
&\le \|(\mu+ \A(t))^{-1}\|^{\gamma}_{\mathcal{L}(\V',\Hi)} \|(\mu+ \A(t))^{-1}\|^{1-\gamma}_{\mathcal{L}(\Hi)}\\
& \le \frac{C}{(\mu+1)^{1-\frac{\gamma}{2}}}.
\end{align*}
We have   $ \|e^{rA(t)}\|_{\mathcal{L}(\Hi)}\leq 1 $ and $\|e^{-rA(t)}\|_{\mathcal{L}(\V', \Hi)} \le \frac{C}{\sqrt{r}}$ (see e.g., \cite{HO15}, 
Proposition 6). Since $\V_{\gamma}'=[\V',\Hi]_{1-\gamma}$ we use interpolation and obtain the third estimate. 
\end{proof}

We make some comments on property [H4]. Let $\gamma \in [0, 1]$ and consider for  fixed $t, s \in [0, \tau]$
\begin{equation}\label{eq2.1}
\A(t) - \A(s) \in \mathcal{L}(\V, \V'_\gamma).
\end{equation}
Obviously, \eqref{eq2.1} holds for all $t, s \in [0, \tau]$ if $\gamma = 1$ since each operator $\A(t)$ is bounded from $\V$ into $\V'$. \\
Observe that \eqref{eq2.1} is equivalent to 
\begin{equation}\label{eq2.2}
| \fra(t,u,v) - \fra(s,u,v) | \le C_{t,s} \| u \|_{\V} \| v \|_{\V_\gamma}
\end{equation}
for some positive constant $C_{t,s}$ and all $u, v \in \V$. Morover, one can take  $C_{t,s} = \| \A(t) - \A(s) \|_{\mathcal{L}(\V, \V'_\gamma)}$. In order to see this, one writes for  $u, v \in \V$
\begin{equation}\label{eqb00}
 \fra(t,u,v) - \fra(s, u, v) = \langle \A(t) u - \A(s)u, v \rangle,
 \end{equation}
and obtains immediately that \eqref{eq2.1} implies \eqref{eq2.2}. For the converse, we note that by \eqref{eqb00},  $v \mapsto  \langle \A(t) u - \A(s)u, v \rangle $ extends to  a (anti-) linear continuous functional on $\V_\gamma$ (for fixed $u \in \V$). The rest of the claim is easy to check. \\
Similarly to the previous remark,  $\A(t) - \A(s)$ extends to a bounded operator from $\V_\gamma$ to $\V'$ {\it if and only if} 
\begin{equation}\label{eq2.3}
| \fra(t,u,v) - \fra(s,u,v) | \le C_{t,s} \| u \|_{\V_\gamma} \| v \|_{\V}
\end{equation}
for al $u, v \in \V$. \\

Our next lemma shows stability of the Kato square root property if \eqref{eq2.2} or \eqref{eq2.3} holds for some $\gamma \in [0, 1)$. 

\begin{lemma}\label{lem2.2} Suppose the assumptions of the previous lemma.  Given $t, s \in [0, \tau]$. Suppose that there exists $\gamma \in [0, 1)$ such that 
either \eqref{eq2.2} or \eqref{eq2.3} holds for $t, s$. If $D(A(s)^{\frac{1}{2}})  =  \V$ then $D(A(t)^{\frac{1}{2}}) = \V$. 
\end{lemma}
\begin{proof}

Suppose that \eqref{eq2.2} is satisfied. Recall that 
 $$ A(t)^{-\frac{1}{2}}=\frac{1}{\pi}\int_{0}^{\infty}\mu^{-\frac{1}{2}} (\mu+A(t))^{-1}\,  d\mu.$$
 Hence
 $$A(t)^{-\frac{1}{2}}-A(s)^{-\frac{1}{2}} =\frac{-1}{\pi}\int_{0}^{\infty}\mu^{-\frac{1}{2}} (\mu+\A(t))^{-1}(\mathcal{A}(t)-\mathcal{A}(s)) (\mu+A(s))^{-1}\,  d\mu.$$
Using the previous lemma, we estimate the $\Hi- \V$ norm of the term in the integral by
 \begin{align*}
 & \mu^{-\frac{1}{2}} \|(\mu+A(t))^{-1}\|_{\mathcal{L}(\V_{\gamma}',\V)} \|(\mu+A(s))^{-1}\|_{\mathcal{L}(\Hi,\V)}
\|\mathcal{A}(t)-\mathcal{A}(s)\|_{\mathcal{L}(\V,\V'_\gamma)}\\
&\le C C_{t,s} \mu^{-\frac{1}{2}} (1+\mu)^{\frac{\gamma}{2}-1}.
\end{align*}
This implies that $A(t)^{-\frac{1}{2}}-A(s)^{-\frac{1}{2}}$ is bounded from $\Hi$ to $\V$ and hence 
$D(A(t)^{\frac{1}{2}}) \subseteq \V$.
In addition, for $u \in D(A(t)^{\frac{1}{2}})$
\begin{align*}
\|  u \|_{\V} &\le \|[A(t)^{-\frac{1}{2}}-A(s)^{-\frac{1}{2}}] A(t)^{\frac{1}{2}} u \|_{\V} +  \|A(s)^{-\frac{1}{2}} A(t)^{\frac{1}{2}} u \|_{\V}\\
&\le \left( \|A(t)^{-\frac{1}{2}}-A(s)^{-\frac{1}{2}}\|_{\mathcal{L}(\Hi,\V)}+\|A(s)^{-\frac{1}{2}}\|_{\mathcal{L}(\Hi,\V)} \right) \|A(t)^{\frac{1}{2}} u \|_{\Hi}\\
& \le C' (1 + C_{t,s}) \|A(t)^{\frac{1}{2}} u \|_{\Hi}.
\end{align*}
Suppose now that $u  \in \V = D(A(s)^{\frac{1}{2}})$. Then we write as before
$$ \A(t)^{\frac{1}{2}} u =\frac{1}{\pi}\int_{0}^{\infty}\mu^{-\frac{1}{2}} A(t)(\mu+A(t))^{-1}u\, d\mu$$
so that 
\begin{align*}
(\A(t)^{\frac{1}{2}}-\A(s)^{\frac{1}{2}}) u &= \frac{1}{\pi}\int_{0}^{\infty}\mu^{-\frac{1}{2}} [ A(t)(\mu+A(t))^{-1}u - A(s) (\mu + A(s))^{-1} u]\,  d\mu\\
&=\frac{-1}{\pi}\int_{0}^{\infty}\mu^{\frac{1}{2}} [ (\mu+A(t))^{-1}u - (\mu + A(s))^{-1} u]\, d\mu\\
&= \frac{1}{\pi}\int_{0}^{\infty}\mu^{\frac{1}{2}}  (\mu+\A(t))^{-1}(\mathcal{A}(t)-\mathcal{A}(s)) (\mu+A(s))^{-1}u\,  d\mu.
\end{align*}
We estimate the norm in $\Hi$ of the term inside the integral by
\begin{align*}
 & \mu^{\frac{1}{2}} \|(\mu+\A(t))^{-1}\|_{\mathcal{L}(\V'_\gamma,\Hi)} \|(\mu+A(s))^{-1}\|_{\mathcal{L}(\V)}
\|\mathcal{A}(t)-\mathcal{A}(s)\|_{\mathcal{L}(\V,\V'_\gamma)} \| u \|_{\V}\\
&\le C C_{t,s} \mu^{\frac{1}{2}} (1+\mu)^{\frac{\gamma}{2}-2}  \| u \|_{\V}.
\end{align*}
This gives $u \in D(A(t)^{\frac{1}{2}})$ and 
$$\| A(t)^{\frac{1}{2}} u \| \le C C_{t,s} \| u \|_{\V} + \| A(s)^{\frac{1}{2}} u \|$$
which proves the lemma. Note that if we assume \eqref{eq2.3} then we argue by duality and prove the lemma for ${A(t)^*}^{\frac{1}{2}}$. It is well known that the   equality  $D({A(t)^*}^{\frac{1}{2}}) = \V$ is equivalent to $D(A(t)^{\frac{1}{2}}) = \V$. 
\end{proof} 

\begin{proposition}\label{prop2.3}  Suppose that [H1]-[H3] are satisfied with $\nu = 0$. Fix $s \in [0, \tau]$ and suppose that either \eqref{eq2.2} or \eqref{eq2.3} holds for  some $\gamma \in [0, 1)$ with a constant $C$ independent of $t \in [0, \tau]$   (i.e., $C_{t,s} \le C$ for all $t$).  If $D(A(s)^{\frac{1}{2}})=\V$,  then  $D(A(t)^{\frac{1}{2}})= \V$ for all $t \in [0, \tau] $ and there exist positive  constants $C_1, C_2$ such that  
\begin{equation}\label{KatoU}
C_1 \| u \|_{\V} \le \| A(t)^{\frac{1}{2}} u \| \le C_2 \| u \|_{\V} \ {\rm  for\  all\ }  u \in \V \ {\rm and\ all \ } t \in [0, \tau]. 
\end{equation}
\end{proposition}
\begin{proof} All the details are  already given in the proof of the previous lemma. 
\end{proof}

\begin{definition}\label{def2.4} We say that $(A(t))$ (or the corresponding forms $\fra(t)$) satisfy the {\it uniform Kato square root property} if 
$D(A(t)^{\frac{1}{2}})= \V$ for all $t \in [0, \tau]$ and there exist constants $C_1, C_2 > 0$ such that \eqref{KatoU} is satisfied. 
\end{definition}

The uniform Kato square root property is obviously satisfied for symmetric forms. It is also satisfied for uniformly elliptic operators (not necessarily symmetric) 
$$A(t) = - \sum_{k,l= 1}^d \partial_k (a_{kl}(t,x) \partial_l)$$
on $L^2(\R^d)$ since $\| \nabla u \|_2 $ is equivalent to $\| A(t)^{\frac{1}{2}} u \|_2$ with constants depending only on the dimension and the ellipticity constants. See \cite{AHLMT}.  \\
The previous proposition says  that in order to have the uniform Kato square root property one needs only to check that $D(A(s)^{\frac{1}{2}}) = \V$ for one $s \in [0, \tau]$ provided \eqref{eq2.2} or \eqref{eq2.3} holds for some $\gamma \in [0, 1)$. 

In the next lemma we show  a quadratic estimate for $A(t)$ with constant independent of $t$. Here we assume  the uniform Kato square root property and give  a short  proof for  the quadratic estimate. It is possible to prove the same result without the uniform Kato square root property by proving that the holomorphic functional calculus of  $A(t)$ has uniform estimate with respect to $t$ (this is the case since the resolvent have uniform estimates). It is well known that quadratic estimates in $\Hi$ are intimately related to the holomorphic functional calculus, see \cite{Mc}.   Quadratic estimates are an important tool in harmonic analysis and we will use them at several places  in the proofs
of maximal regularity.  

\begin{lemma}\label{lem2.4}
Suppose in addition to [H1]-[H3] (with $\nu = 0$) that the uniform Kato square root property is satisfied. Then there exists a constant $C$ such that 
for every $t \in [0, \tau]$
\begin{equation}\label{eq2.4}
\int_0^\tau \| A(t)^{\frac{1}{2}} e^{-s A(t)} x \|^2 \, ds \le C  \| x \|^2
\end{equation}
for all $x \in \Hi$.
\end{lemma}
\begin{proof} 
By the uniform Kato square root property, we have
\begin{eqnarray*}
\int_0^\tau \| A(t)^{\frac{1}{2}} e^{-s A(t)} x \|^2 \, ds &\le& C_2 \int_0^\tau \| e^{-s A(t)} x \|_{\V}^2 \, ds\\
&\le& C' \int_0^\tau \Re \fra(e^{-s A(t)} x, e^{-sA(t)} x)  \, ds\\ 
&=& C' \int_0^\tau \Re ( A(t) e^{-s A(t)} x, e^{-sA(t)} x)  \, ds\\
&=& - \frac{C'}{2}  \int_0^\tau \frac{d}{ds} \| e^{-sA(t)} x \|^2 \, ds\\
&=& \frac{C'}{2} [ \| x \|^2 - \| e^{-\tau A} x \|^2] \le \frac{C'}{2 }  \| x \|^2.
\end{eqnarray*}
This proves the lemma.
\end{proof}

We note that the constant $C$ is also independent of $\tau$. We could formulate the lemma with $\int_0^\infty \| A(t)^{\frac{1}{2}} e^{-s A(t)} x \|^2 \, ds$. Let us also mention the following $L^p$-version. 

\begin{lemma}\label{lem2.5}
Suppose the assumptions of the previous lemma. Let $p \ge 2$. Then there exists a constant $C_p$ such that 
\begin{equation}\label{eq2.5}
\int_0^\tau \| A(t)^{1/p} e^{-s A(t)} x \|^p \, ds \le C_p  \| x \|^p
\end{equation}
for all $x \in \Hi$.
\end{lemma}
\begin{proof}
We fix $t \in [0, \tau]$ and $s > 0$.  We define 
$$F(z) x := A(t)^{z/2} e^{-s A(t)} x.$$
It is  a classical fact that $F$ is a holomorphic function on $\C^+$. In addition, each operator $A(t)$, as an accretive operator on $\Hi$,  has bounded imaginary powers. That is $\| A^{is} \|_{\mathcal{L}(\Hi)} \le C$ for some constant $C$ and all $s \in \R$.  See \cite{Mc}. 
Using this one obtains immediately that 
$$ \| F(is) x \|_{L^\infty(0, \tau; \Hi)} \le C \| x \|.$$ 
On the other hand, by Lemma \ref{lem2.4} and again uniform boundedness of imaginary powers on $\Hi$ we obtain
$$ \| F(1+ is) x \|_{L^2(0, \tau; \Hi)} \le  C \| x \|.$$ 
We apply Stein's complex interpolation theorem  to obtain that  for all $p\geq 2$
$$ \| F (\tfrac{2}{p}) x \|_{L^p(0, \tau; \Hi)} \le  C_p \| x \|$$
and we obtain the lemma.
\end{proof}

Let $u$ be the solution of (P') by Lions' theorem. Lions also proved that $u \in C([0, \tau]; \Hi)$. Since $u \in L^2(0, \tau; \V)$ we have $u(t) \in \V$ for a.e. $t \in [0, \tau]$. It is very useful to know whether  $u(t) \in \V$ for every $t \in [0, \tau]$. We prove this in the following lemma under an additional assumption that $u \in L^\infty(0, \tau; \V)$. We shall see later that this property holds when we assume that 
\eqref{eqHyp} is satisfied. 

\begin{lemma}\label{lem2.6} 
Suppose [H1]-[H3] and suppose in addition  that the solution $u$ belongs to $L^\infty(0, \tau; \V)$. Then $u(t) \in \V$ for every $t \in [0, \tau]$.
\end{lemma}
\begin{proof} 
Let $\Gamma=\{t\in [0,\tau]  \    \mbox{s.t.}\  u(t)\notin \V \}$.  Since $\Gamma$ has measure zero,  $[0, \tau]\setminus \Gamma$ is dense in $[0, \tau]$. Suppose that  $t\in \Gamma$.  There exists a sequence  $(t_{n})_{n}\in [0,\tau]\setminus \Gamma$ such that $t_{n}\rightarrow t$ as $n \to \infty$.
Since the sequence $(u(t_{n}))$ is bounded in $\V$ we can exact a subsequence $u(t_{n_k})$ which converges weakly to some $v$ 
in $\V$. By continuity of $u$ in $\Hi$, $u(t_{n_k})$ converges (in $\Hi$)  to $u(t)$. 
 This gives $u(t) = v \in \V$.  Hence $\Gamma$ is empty.
\end{proof}

\section{Key estimates}\label{sec3}

In this section we state and prove some other  estimates which we will need in the proofs of the main results.
\medskip

\begin{lemma}\label{lem3.1}
Suppose in addition to [H1]-[H3] (with $\nu = 0$) that the uniform Kato square root property is satisfied.  Define 
$$L_0(f)(t) := \int_0^t e^{-(t-s) A(t)} f(s) \, ds.$$
Then $L_0 : L^2(0, \tau; \Hi) \to L^\infty(0, \tau; \V)$ is  a bounded operator. 
\end{lemma}
\begin{proof}
By the uniform Kato square root property,
$$ \| L_0(f)(t) \|_{\V} \le C_2 \| A(t)^{\frac{1}{2}} \int_0^t  e^{-(t-s)A(t)} f(s)\, ds \|.$$
On the other hand, for $x \in \Hi$
\begin{eqnarray*}
&& \hspace{-.7cm}| (A(t)^{\frac{1}{2}} \int_0^t  e^{-(t-s)A(t)} f(s)\, ds, x) |\\ 
&=& \int_0^t (f(s), {A(t)^*}^{\frac{1}{2}} e^{-(t-s)A(t)^*} x)\,  ds\\
&\le& \|f \|_{L^2(0, \tau; \Hi)} \left( \int_0^t  \| {A(t)^*}^{\frac{1}{2}} e^{-(t-s)A(t)^*} x \|^2 \, ds \right)^{\frac{1}{2}}\\
&\le& C \|f \|_{L^2(0, \tau; \Hi)} \| x \|.
\end{eqnarray*}
The latest inequality follows from Lemma \ref{lem2.4} applied to the adjoint operator $A(t)^*$ (note that $A(t)^*$ satisfies  the same properties  as $A(t)$). The constant $C$ is independent of $t$.
Therefore,
\begin{equation}\label{eq3.1}
\| L_0(f)(t) \|_{\V} \le C_2 C \|f \|_{L^2(0, \tau; \Hi)}. 
\end{equation}
This implies immediately that $L_0 : L^2(0, \tau; \Hi) \to L^\infty(0, \tau; \V)$ is bounded. 
\end{proof}

Now we study boundedness on $L^2(0, \tau; \Hi)$ of the operator
\begin{equation}\label{eq3.2}
L (f)(t) := \int_0^t A(t) e^{-(t-s) A(t)} f(s)\, ds.
\end{equation}
It is proved in \cite{HO15} that $L$ is bounded on $L^p(0, \tau; \Hi)$ for all $p \in (1, \infty)$ provided 
 $t \mapsto \fra(t,.,.)$ is $C^\epsilon$ for some $\epsilon > 0$ (or similarly, $t \mapsto \mathcal{A}(t)$ is 
$C^\epsilon$ on $[0, \tau]$ with values in $\mathcal{L}(\V, \V')$). The proof for the case $p=2$ is based on vector-valued pseudo-differential operators. The extension from $p = 2$ to $p \in (1, \infty)$ uses H\"ormander's almost $L^1$-condition for singular integral operators. Here we give a direct proof for the case $p=2$ which does not appeal to pseudo-differential operators. It is essentially  based on the quadratic estimate of Lemma \ref{lem2.4}.
 
\begin{proposition}\label{HO1} Suppose [H1]-[H3] (with $\nu = 0$) and the uniform Kato square root property. Let $\gamma \in (0, 1]$.
If 
\begin{equation}\label{eqLL2}
\sup_{s\in [0,\tau]}\int_{s}^{\tau}\frac{\|\mathcal{A}(t)-\mathcal{A}(s)\|_{\mathcal{L}(\V,\V_{\gamma}')}^{2}}{|t-s|^{{\gamma}}}\, dt <\infty,
\end{equation}
then the  operator $L$ is bounded on $L^{2}(0,\tau; \Hi)$. 
\end{proposition}

\begin{proof}
Fix $\gamma \in [0, 1]$. Take $g\in L^{2}(0,\tau;\Hi)$. We have 
\begin{align*}
& |\int_{0}^{\tau}\int_{0}^{t}(A(t)e^{-(t-s)A(t)}f(s),g(t))\, dsdt|\\
& = |\int_{0}^{\tau} \int_{0}^{t} \left( A(t)^{\frac{1}{2}}e^{-\frac{(t-s)}{2}A(t)}f(s), {A(t)^*}^{\frac{1}{2}}e^{-\frac{(t-s)}{2}A(t)^*}g(t) \right)\, ds dt|\\
&\le \int_{0}^{\tau}\int_{0}^{t} \|A(t)^{\frac{1}{2}}e^{-\frac{(t-s)}{2}A(t)}f(s)\| \|{A(t)^*}^{\frac{1}{2}}e^{-\frac{(t-s)}{2}A(t)^{*}}g(t)\| \, ds dt \\
&\le  \int_{0}^{\tau} \left(\int_{0}^{t} \|A(t)^{\frac{1}{2}}e^{-\frac{(t-s)}{2}A(t)}f(s)\|^{2} \, ds \right)^{\frac{1}{2}}  \left(\int_{0}^{t} \|{A(t)^*}^{\frac{1}{2}}e^{-\frac{(t-s)}{2}A(t)^{*}}g(t)\|^{2} \, ds \right)^{\frac{1}{2}} \,dt\\
& \le C \int_{0}^{\tau} \left(\int_{0}^{t} \|A(t)^{\frac{1}{2}}e^{-\frac{(t-s)}{2}A(t)}f(s)\|^{2} \, ds \right)^{\frac{1}{2}}\|g(t)\| \, dt\\
&\le C \left(\int_{0}^{\tau}\int_{0}^{t}\|A(t)^{\frac{1}{2}}e^{-\frac{(t-s)}{2}A(t)}f(s)\|^{2} \, dsdt \right)^{\frac{1}{2}} \| g \|_{L^2(0, \tau; \Hi)}.
\end{align*}
Here we use the quadratic estimate of Lemma \ref{lem2.4} for the adjoint operator $A(t)^*$. It follows that
for all $f \in L^2(0, \tau; \Hi)$
\begin{equation}\label{eq3.3}
\| L(f) \|_{L^2(0, \tau; \Hi)} \le C \left(\int_{0}^{\tau}\int_{0}^{t}\|A(t)^{\frac{1}{2}}e^{-\frac{(t-s)}{2}A(t)}f(s)\|^{2} \, dsdt \right)^{\frac{1}{2}}.
\end{equation}
We use again Lemma \ref{lem2.4} and obtain
\begin{align*}
&  \int_{0}^{\tau}\int_{0}^{t}\|A(t)^{\frac{1}{2}}e^{-(t-s)A(t)}f(s)\|^{2} \, dsdt\\
&\leq 2  \int_{0}^{\tau}\int_{0}^{t}\|A(t)^{\frac{1}{2}}e^{-(t-s)A(t)}f(s)-A(s)^{\frac{1}{2}}e^{-(t-s)A(s)}f(s)\|^{2} \,dsdt \\
&+\   2 \int_{0}^{\tau}\int_{s}^{\tau}\|A(s)^{\frac{1}{2}}e^{-(t-s)A(s)}f(s)\|^{2} \, dtds \\
&\le 2  \int_{0}^{\tau}\int_{0}^{t}\|A(t)^{\frac{1}{2}}e^{-(t-s)A(t)}f(s)-A(s)^{\frac{1}{2}}e^{-(t-s)A(s)}f(s)\|^{2} \,dsdt \\
&+ 2 C \int_{0}^{\tau}\|f(s)\|^{2} \,ds. 
\end{align*}
 We choose a  contour $\Gamma$ in the positive half-plane and we  write by the holomorphic functional calculus 
\begin{align*}
&A(t)^{\frac{1}{2}}e^{-(t-s)A(t)}f(s)-A(s)^{\frac{1}{2}}e^{-(t-s)A(s)}f(s)\\
&= \int_\Gamma \lambda^{\frac{1}{2}}e^{-(t-s)\lambda} [ (\lambda - \mathcal{A}(t))^{-1} - (\lambda - \mathcal{A}(s))^{-1} ] \, d\lambda \,  f(s)\\
&=\int_{\Gamma} \lambda^{\frac{1}{2}}e^{-(t-s)\lambda}(\lambda-\mathcal{A}(t))^{-1}(\mathcal{A}(t)-\mathcal{A}(s))(\lambda-\mathcal{A}(s))^{-1}f(s) \, d\lambda.
\end{align*}
We estimate the norm in $\Hi$ of the latest term. For $\lambda = | \lambda | e^{i \theta}$ we apply Lemma \ref{lem2.1} and obtain
\begin{align*}
& \| \int_{\Gamma} \lambda^{\frac{1}{2}}e^{-(t-s)\lambda}(\lambda-\mathcal{A}(t))^{-1}(\mathcal{A}(t)-\mathcal{A}(s))(\lambda-\mathcal{A}(s))^{-1}f(s) \, d\lambda \| \\
&\le \int_{\Gamma} | \lambda |^{\frac{1}{2}}e^{-(t-s)|\lambda| \cos \theta} \| (\lambda-\mathcal{A}(t))^{-1} \|_{\mathcal{L}(\V_\gamma', \Hi)}
 \| \mathcal{A}(t) - \mathcal{A}(s) \|_{\mathcal{L}(\V, \V'_\gamma)} \\
 & \times  \| (\lambda-\mathcal{A}(s))^{-1} \|_{\mathcal{L}( \Hi, \V)} \, |d\lambda| \, \| f(s)\|\\
  &\le C \int_\Gamma | \lambda |^{\frac{1}{2}}e^{-(t-s)|\lambda| \cos \theta} \frac{1}{ (1 + | \lambda |)^{1 - \gamma/2}}
  \frac{1}{ (1 + | \lambda |)^{\frac{1}{2}}} \| \mathcal{A}(t) - \mathcal{A}(s) \|_{\mathcal{L}(\V, \V'_\gamma)}\, |d\lambda| \,  \|f(s)\|\\
  &\le C' \frac{\| \mathcal{A}(t) - \mathcal{A}(s) \|_{\mathcal{L}(\V, \V'_\gamma)}}{ | t-s|^{\gamma/2}} \| f(s) \|.
  \end{align*}
 Here and at other places we use the estimate
 \begin{equation}\label{eq101}
 \int_0^\infty \frac{e^{-r(t-s)}}{(1+r)^{1-\frac{\gamma}{2}}} \, dr \le \frac{C}{ (t-s)^{ \frac{\gamma}{2}}}
 \end{equation}
 for some constant $C$ and all $s < t$. This is seen by making the change of the variable $v := r(t-s)$ in the LHS which then coincides with
 \[
\frac{1}{ (t-s)^{ \frac{\gamma}{2}}}  \int_0^\infty \frac{e^{-v}}{ (t-s + v)^{1-\frac{\gamma}{2}}} \, dv.
  \]
  The latter term is bounded by 
   \[
\frac{1}{ (t-s)^{ \frac{\gamma}{2}}}  \int_0^\infty \frac{e^{-v}}{ v^{1-\frac{\gamma}{2}}} \, dv = \frac{C}{ (t-s)^{ \frac{\gamma}{2}}}.
  \]
The previous estimates give
\begin{align}
&\int_{0}^{\tau}\int_{0}^{t}\|A(t)^{\frac{1}{2}}e^{-(t-s)A(t)}f(s)-A(s)^{\frac{1}{2}}e^{-(t-s)A(s)} f(s)\|^{2} \, dsdt \nonumber\\
&\le C' \int_{0}^{\tau}\int_{s}^{\tau}\frac{\|\mathcal{A}(t)-\mathcal{A}(s)\|_{\mathcal{L}(V,V_{\gamma}')}^{2}}{|t-s|^{{\gamma}}} \, dt \, \|f(s)\|^{2} \, ds \nonumber\\
&\le C' \sup_{s\in [0,\tau]} \int_{s}^{\tau}\frac{\|\mathcal{A}(t)-\mathcal{A}(s)\|_{\mathcal{L}(\V,\V_{\gamma}')}^{2}}{|t-s|^{{\gamma}}} \, dt \int_{0}^{\tau}\|f(s)\|^{2} \, ds.\label{estHO1}
\end{align}  
This proves that  $L$ is bounded on $L^2(0, \tau; \Hi)$. \end{proof}

\begin{remark} One may argue as in the proof of Lemma 11 in \cite{HO15} and obtain boundedness of $L$ on $L^p(0, \tau; \Hi)$ at least for 
$p \in (1, 2)$. 
\end{remark}

\begin{corollary}\label{cor3.61}
1) If $\mathcal{A}$ satisfies \eqref{eqHyp} then $L$ is bounded on $L^2(0, \tau; \Hi)$.\\
2) If [H4] is satisfied for some $\gamma \in (0, 1)$ then $L$ is bounded on $L^2(0, \tau; \Hi)$.  
\end{corollary}
\begin{proof}
Assertion 1) follows  directly from Proposition \ref{HO1} by noticing that  \eqref{eqHyp} implies \eqref{eqLL2} with $\gamma = 1$.
For assertion 2) one uses [H4] to obtain
\[
\int_{s}^{\tau}\frac{\|\mathcal{A}(t)-\mathcal{A}(s)\|_{\mathcal{L}(\V,\V_{\gamma}')}^{2}}{|t-s|^{{\gamma}}} \, dt \le M_\gamma^2 \int_s^\tau
{|t-s|^{{-\gamma}}} \, dt \le c M_\gamma^2
\]
for some constant $c > 0$. The result follows from Proposition \ref{HO1}.
\end{proof}
Note that if $\gamma \in (0, 1)$ we do not require any regularity property for $\mathcal{A}$ in assertion 2) of the previous proposition.

\begin{proposition}\label{YA1}
Suppose  [H1]-[H3] (with $\nu = 0$) and  the uniform Kato square root property.  Let $f \in L^2(0, \tau; \Hi)$, $u_{0}\in V$  and 
let $u$ be the Lions' solution to the problem (P').\\
1) If $\mathcal{A}$  satisfies \eqref{eqHyp} then $u\in L^{\infty}(0,\tau; \V)$ and there exists a constant $C$ independent of $u_0$ and $f$ such that
\begin{equation}\label{eq3.4}
\| u \|_{L^\infty(0, \tau ; \V)} \le C \left[ \| u_0 \|_{\V} + \| f \|_{L^2(0, \tau; \Hi)} \right].
\end{equation}
2) Let $\gamma \in [0, 1)$ and suppose that [H4] is satisfied. Then $u\in L^{\infty}(0,\tau; \V)$ and \eqref{eq3.4} holds.
\end{proposition}

\begin{proof}
As we already mentioned above, by Lions' theorem  there exists a  unique solution $u$ to the problem (P') such that $u \in H^{1}(0,\tau; \V')\cap L^{2}(0,\tau ; \V)$.   For  $0\leq s \leq t \leq \tau$,  we define $v(s) :=e^{-(t-s)A(t)} u(s)$.  We write  $v(t)=v(0)+\int_{0}^{t} v'(s)\, ds$ and obtain as in \cite{HO15} (Lemma 8)
 \begin{align} \label{F1}
  u(t)&=e^{-tA(t)}u_{0}+\int_{0}^{t}e^{-(t-s)\mathcal{A}(t)}(\mathcal{A}(t)-\mathcal{A}(s))u(s) \, ds \nonumber \\&+\int_{0}^{t}e^{-(t-s)A(t)}f(s) \, ds \\
  &=: R_0 u_{0}(t) +S_0 u(t)+ L_0f(t).\nonumber 
   \end{align}
  Clearly  there exists a constant $C> 0$ such that   for   all $ u_{0} \in \V$, 
\begin{equation}\label{eq3.5}
\| R_0 u_{0}(t) \|_{\V} = \|e^{-tA(t)}u_{0}\|_{\V} \le C \|u_{0}\|_{\V}.
\end{equation}
By Lemma \ref{lem3.1}, 
\begin{equation}\label{eq3.5L0}
\| L_0 f(t) \|_{\V} \le C \| f \|_{L^2(0, \tau; \Hi)}.
\end{equation}
  Next we prove that $S_0 \in \mathcal{L}(L^{\infty}(0,\tau; \V))$.
  Let $ g \in L^{\infty}(0,\tau; \V)$. We have by the uniform Kato square root property
  $$ \| S_0 g (t) \|_{\V} \le C_2 \|  \int_0^t A(t)^{\frac{1}{2}}  e^{-(t-s)A(t)}(\mathcal{A}(t)-\mathcal{A}(s))g(s)\,  ds \|.$$
  In order to estimate the term on the RHS we argue as in the proof of Lemma \ref{lem3.1} and use Lemma \ref{lem2.1}.   For $x \in \Hi$  and $\gamma \in [0,1]$, we have
 \begin{align*}
 &| (\int_{0}^{t} A(t)^{\frac{1}{2}} e^{-(t-s)\mathcal{A}(t)}(\mathcal{A}(t)-\mathcal{A}(s)) g(s) \, ds, x ) | \\
 &=| \int_{0}^{t} ( e^{-\frac{(t-s)}{2}\mathcal{A}(t)} (\mathcal{A}(t)-\mathcal{A}(s)) g(s),  {A(t)^*}^{\frac{1}{2}} e^{-\frac{(t-s)}{2}A(t)^*}x) \, ds| \\
  &\le \left(\int_{0}^{t} \| e^{-\frac{(t-s)}{2}\mathcal{A}(t)} (\mathcal{A}(t)-\mathcal{A}(s)) g(s) \|^2 \, ds \right)^{\frac{1}{2}} \left(\int_0^t \|{A(t)^*}^{\frac{1}{2}} e^{-\frac{(t-s)}{2}A(t)^*}x \|^2 \, ds \right)^{\frac{1}{2}}\\
& \le C \| x \| \left(\int_{0}^{t}\|e^{-\frac{(t-s)}{2}\mathcal{A}(t)}\|^2_{\mathcal{L}(\V_{\gamma}',\Hi)} \|\mathcal{A}(t)-\mathcal{A}(s)\|^2_{\mathcal{L}(\V,\V_{\gamma}')} \, ds \right)^{\frac{1}{2}} \|g\|_{L^{\infty}(0,t;\V)}\\
&\le C \| x \| \left( \int_{0}^{t}\frac{\|\mathcal{A}(t)-\mathcal{A}(s)\|^2_{\mathcal{L}(\V,\V_{\gamma}')}}{(t-s)^{\gamma}} \, ds \right)^{\frac{1}{2}}\|g\|_{L^{\infty}(0,t;\V)}.
\end{align*}
Therefore,
\begin{equation}\label{eq1000}
 \| S_0 g (t) \|_{\V}  \le C C_2 \left(\int_{0}^{t}\frac{\|\mathcal{A}(t)-\mathcal{A}(s)\|^2_{\mathcal{L}(\V,\V_{\gamma}')}}{(t-s)^{\gamma} } \, ds \right)^{\frac{1}{2}}\|g\|_{L^{\infty}(0,\tau;\V)}.
 \end{equation}
Suppose $\gamma = 1$. It follows from the  assumption \eqref{eqHyp} that $S_0$ is a bounded operator on $L^{\infty}(0,\tau;\V)$ with
\begin{equation}\label{eq10001}
\| S_0 \|_{\mathcal{L}(L^\infty(0, \tau; \V))} \le C C_2 \left( \sup_{t \in [0, \tau]} \int_{0}^{t}\frac{\|\mathcal{A}(t)-\mathcal{A}(s)\|^2_{\mathcal{L}(\V,\V')}}{t-s } \, ds \right)^{\frac{1}{2}}.
\end{equation}

In order to continue we wish to take the inverse of $I- S_0$. Let $\varepsilon > 0$ and  $\tau_1$ be as \eqref{eqHyp}. We work on the interval 
$[0, \tau_1]$. We have 
\[
\sup_{t \in [0, \tau_1]} \int_{0}^{\tau_1}\frac{\|\mathcal{A}(t)-\mathcal{A}(s)\|^2_{\mathcal{L}(\V,\V')}}{t-s } \, ds 
< \varepsilon.
\]
It follows from \eqref{eq10001} that  $\| S_0 \|_{\mathcal{L}(L^\infty(0, \tau_1; \V))} < 1$. 

Therefore, 
$ u = (I- S_0)^{-1} (R_0 u_0 + L_0 f)$ and  we obtain from \eqref{eq3.5} and \eqref{eq3.5L0} that $u \in L^\infty(0, \tau_1; \V)$ and 
\eqref{eq3.4} is satisfied on $[0, \tau_1]$. Now repeat the same strategy. We use \eqref{eqHyp}, we work on $[\tau_{i-1}, \tau_i]$ and argue exactly as before. 
We obtain  \eqref{eq3.4}  on each  sub-intervals $[\tau_{i-1}, \tau_i]$. This implies \eqref{eq3.4} on $[0, \tau]$ for arbitrary $\tau > 0$ and finishes the proof of assertion 1).\\
In order to prove assertion 2), we use [H4] and  \eqref{eq1000}. We obtain
\[
\| S_0 g(t) \|_{\V} \le C C_2 M_\gamma \left( \int_0^t  \frac{1}{(t-s)^\gamma} \, ds \right)^{\frac{1}{2}} \|g\|_{L^{\infty}(0,\tau;\V)} \le C' \tau^{\frac{1 -\gamma}{2}}\|g\|_{L^{\infty}(0,\tau;\V)}. 
\]
We see that $\| S_0 \|_{\mathcal{L}(L^\infty(0, \tau; \V))} < 1$ for small $\tau > 0$. We split $[0, \tau]$ into a finite number of intervals with small sizes and then  argue as previously.  
\end{proof}

\begin{proposition}\label{propV}
Suppose  the assumptions of the previous proposition. Then $u(t) \in \V$ for every $t \in [0, \tau]$.
\end{proposition}
\begin{proof} This is an application of Lemma \ref{lem2.6}  and Proposition \ref{YA1}. Note that $u(t)$ is well defined for every $t$ since $u \in C([0, \tau], \Hi)$ by Lions' theorem.
\end{proof}

The following lemma was first proved in \cite{HO15} under the assumption that $\mathcal{A}(.) \in C^\alpha(0, \tau; \mathcal{L}(\V, \V'))$ for some $\alpha > \frac{1}{2}$. See also \cite{Ou15}. We prove it here in the case where 
$\mathcal{A}$ satisfies \eqref{eqHyp} and for arbitrary 
$\mathcal{A}$ if [H4] is satisfied for some $\gamma \in (0, 1)$. 

\begin{lemma}\label{lem3.33} 
Suppose [H1]-[H3] (with $\nu = 0$) and the uniform Kato square root property.  Define the operator 
$$R u_{0} (t) := A(t)e^{-tA(t)}u_{0}$$
for $u_{0} \in \V$ and $ t \in [0,\tau]$.\\ 
1) Suppose \eqref{eqHyp}. Then $R $ is bounded from $\V$ into $L^{2}(0,\tau; \Hi)$. \\
2) If [H4] is satisfied for some $\gamma \in (0, 1)$ then $R $ is bounded from $\V$ into $L^{2}(0,\tau; \Hi)$. 
\end{lemma}
\begin{proof}
We write
\begin{align*}
R u_{0} (t)&=[A(t)e^{-tA(t)}-A(0)e^{-tA(0)}]u_{0}+A(0)e^{-tA(0)}u_{0}\\
&=: R_{1}u_{0} (t)+R_{2}u_{0} (t).
\end{align*}
We use Lemma \ref{lem2.4}  to obtain 
\begin{eqnarray*}
\|R_{2}u_{0}\|^{2}_{L^{2}(0,\tau;\Hi)}
&=& \int_{0}^{\tau}\|A(0)^{\frac{1}{2}}e^{-tA(0)}A(0)^{\frac{1}{2}}u_{0}\|^{2} \, dt\\
&\le& C \|A(0)^{\frac{1}{2}}u_{0}\|^{2} \le C' \|u_{0}\|_{\V}^{2}.
\end{eqnarray*}
We  estimate $R_{1}u_{0}$. We argue as in the proof of Proposition \ref{HO1}. By the holomorphic functional calculus for the sectorial operators $A(t)$ and $A(0)$ we have 
\begin{align*}
 R_{1}u_{0} (t)=\int_{\Gamma}\lambda e^{-\lambda t}(\lambda-A(t))^{-1}(\mathcal{A}(t)-\mathcal{A}(0))(\lambda-A(0))^{-1} u_0 \, d\lambda.   
\end{align*}
Now taking the norm in $\Hi$ we have 
\begin{align*}
&\|R_{1}u_{0}(t)\|\le \int_{\Gamma}|\lambda| e^{-t \Re \lambda } \|(\lambda-A(t))^{-1}\|_{\mathcal{L}(\V_{\gamma}',\Hi)} \\ &\times \|\mathcal{A}(0)-\mathcal{A}(t)\|_{\mathcal{L}(\V,\V_{\gamma}')} \|(\lambda-A(0))^{-1}\|_{\mathcal{L}(\V)} \, |d\lambda| \, \|u_{0}\|_{\V}\\
&\le C \frac{ \|\mathcal{A}(0)-\mathcal{A}(t)\|_{\mathcal{L}(\V,\V_{\gamma}')}}{t^{\frac{\gamma}{2}}}\|u_{0}\|_{\V}.
\end{align*}
The last estimate follows as in the proof of Proposition \ref{HO1} in which we also use \eqref{eq101}. It is valid for 
$\gamma \in (0,1]$. \\
Therefore
\begin{equation}\label{eq1001}
\|R_{1}u_{0}\|^{2}_{L^{2}(0,\tau;\Hi)} \le C^2 \int_{0}^{\tau}\frac{ \|\mathcal{A}(0)-\mathcal{A}(t)\|^{2}_{\mathcal{L}(\V,\V_\gamma')}}{t^{{\gamma}}} \, dt \, \|u_{0}\|^{2}_{\V}.
\end{equation}
For $\gamma = 1$ we use the assumption \eqref{eqHyp} and  obtain 
$$\|R_1 u_{0}\|^{2}_{L^{2}(0,\tau;\Hi)} \le C'  \|u_{0}\|^{2}_{\V}$$
for some constant $C' > 0$. 
This proves  assertion 1) of the proposition. Assertion 2) follows directly from \eqref{eq1001} when $\gamma \in (0, 1)$. 
\end{proof}

\section{Proofs of the main results}\label{sec4}

After the auxiliary results of the last two sections we are now ready to give the proofs of the main results of this paper. Note that we may assume without loss of generality that [H3] holds with $\nu = 0$, see the beginning of Section \ref{sec2}. 

\begin{proof}[Proof of Theorems \ref{thm1} and \ref{thm2}]

 Let $\gamma  \in (0, 1]$. We give the proof for the two theorems without considering separately the cases $\gamma = 1$ and $\gamma \in (0, 1)$. If $\gamma = 1$ we assume \eqref{eqHyp}. 

Suppose first that $\mathcal{A} \in \dot{H}^{\frac{\gamma}{2}}(0, \tau; \mathcal{L}(\V, \V_\gamma'))$.\\
Let $f \in L^2(0, \tau; \Hi)$ and $u_0 \in \V$. Let $u \in L^2(0, \tau; \V) \cap H^1(0, \tau; \V')$ be the solution of (P')  given by Lions' theorem. Our aim is to prove that $u \in H^1(0, \tau; \Hi)$ or equivalently that $A(.) u(.) \in L^2(0, \tau; \Hi)$. Using \eqref{F1}  we have 
\begin{align*}
  A(t)u(t)&=A(t)e^{-tA(t)}u_{0}+ A(t) \int_{0}^{t} e^{-(t-s)A(t)}(\mathcal{A}(t)-\mathcal{A}(s))u(s) \, ds\\&+A(t)\int_{0}^{t}e^{-(t-s)A(t)}f(s) \, ds\\
  &=: R u_{0}(t) +(S u)(t)+(Lf)(t).
  \end{align*}
  By Lemma \ref{lem3.33} the operator $R $ is bounded from $\V$ into $L^{2}(0,\tau; \Hi)$ and by 
 Corollary \ref{cor3.61} the operator 
$L$ is bounded  on  $L^{2}(0,\tau; \Hi)$. 
Concerning the operator $S$, we have
\begin{equation}\label{eq4.1}
\| S u \|_{L^2(0, \tau; \Hi)} \le C \| u \|_{L^\infty(0, \tau; \V)}
\end{equation}
for some constant $C$ independent of $f$ and $u_0$. \\
Suppose for a moment that \eqref{eq4.1} is proved. Then we apply Proposition \ref{YA1} together with the properties of $R$ and $L$ we just mentioned above and obtain
$A(.)u(.) \in L^2(0, \tau; \Hi)$ with
\begin{eqnarray*}
\| A(.) u(.) \|_{L^2(0, \tau; \Hi)} &\le& \| R u_0 \|_{L^2(0, \tau; \Hi)} + \| S u \|_{L^2(0, \tau; \Hi)} + \| L f \|_{L^2(0, \tau; \Hi)} \\
&\le& C \left( \| u_0 \|_{\V} + \| u \|_{L^\infty(0, \tau; \V)}+ \| f \|_{L^2(0, \tau; \Hi)} \right)\\
&\le& C' \left( \| u_0 \|_{\V} + \| f \|_{L^2(0, \tau; \Hi)} \right).
\end{eqnarray*}
This proves the two theorems in the case where $\mathcal{A} \in \dot{H}^{\frac{\gamma}{2}}(0, \tau; \mathcal{L}(\V, \V_\gamma'))$ for  some $\gamma \in (0, 1]$.  

Now we prove \eqref{eq4.1}.  
We have 
\begin{align*}
& \| S u(t) \|\\
& = \| A(t) \int_{0}^{t}  e^{-(t-s)A(t)}(\mathcal{A}(t)-\mathcal{A}(s))u(s)\, ds \| \\
&\le  \sup_{x \in \Hi, \| x \| = 1} \int_0^t |(A(t)^{\frac{1}{2}} e^{- \frac{(t-s)}{2} A(t)} (\mathcal{A}(t)-\mathcal{A}(s))u(s), 
{A(t)^*}^{\frac{1}{2}} e^{- \frac{(t-s)}{2} A(t)^*} x)| \, ds\\
&\le C \left( \int_0^t  \|A(t)^{\frac{1}{2}} e^{- \frac{(t-s)}{2} A(t)} (\mathcal{A}(t)-\mathcal{A}(s))u(s)\|^2 \, ds \right)^{\frac{1}{2}}.
\end{align*}
Here we use again the quadratic estimate of Lemma \ref{lem2.4}. By analyticity of the semigroup together with Lemma \ref{lem2.1} 
 we have 
\begin{align*}
\|A(t)^{\frac{1}{2}} e^{- \frac{(t-s)}{2} A(t)} (\mathcal{A}(t)-\mathcal{A}(s))u(s)\|^2
& \le \frac{C}{t-s} \|e^{- \frac{(t-s)}{4} A(t)} (\mathcal{A}(t)-\mathcal{A}(s))u(s)\|^2\\
&\le C' \frac{\| \mathcal{A}(t)-\mathcal{A}(s) \|_{\mathcal{L}(\V, \V_\gamma')}^{2}}{|t-s|^{1+\gamma}} \|u(s) \|_{\V}^2.
\end{align*}
Therefore,
$$\| S u(t) \| \le C'' \| u \|_{L^\infty(0, \tau; \V)} \left( \int_0^t \frac{\| \mathcal{A}(t)-\mathcal{A}(s) \|^2_{\mathcal{L}(\V, \V_\gamma') }} {|t-s|^{1+\gamma} } \, ds \right)^{\frac{1}{2}}.$$
This gives 
$$ \| S u \|_{L^2(0, \tau; \Hi)} \le C''  \| u \|_{L^\infty(0, \tau; \V)} \| \mathcal{A} \|_{\dot{H}^{\frac{\gamma}{2}}(0, \tau; \mathcal{L}(\V, \V_\gamma'))}$$
and  finishes the proof of \eqref{eq4.1}.

\medskip

Suppose now that $\mathcal{A}$ is piecewise in $\dot{H}^{\frac{\gamma}{2}}(0, \tau; \mathcal{L}(\V, \V_\gamma'))$.  Then 
$[0, \tau] = \cup_{i=1}^n [\tau_{i-1}, \tau_i]$ and  the restriction of $\mathcal{A}$ to each sub-interval is in $\dot{H}^{\frac{\gamma}{2}}$. We apply the previous proof to each sub-interval and obtain a solution $u_i$ in the  sub-interval $[\tau_{i-1}, \tau_i]$  which has maximal regularity and satisfies  apriori estimates. By 
Proposition \ref{propV}, $u_i(\tau_i) \in \V$ and hence we can solve $u_{i+1}'(t) + A(t) u_{i+1}(t) = f(t) $ on $[\tau_i, \tau_{i+1}]$ with $u_{i+1} (\tau_i) = u_i (\tau_i)$ and 
$u_{i+1}$ has maximal regularity and apriori estimate on $[\tau_i, \tau_{i+1}]$. Now we 
 "glue" the  solutions $u_i$  and obtain a solution $u$ of  (P) on $[0, \tau]$ such that  $u  \in H^1(0, \tau; \Hi)$.   The apriori estimate
\eqref{eq-apriori} on $[0, \tau]$  follows by summing the corresponding apriori estimates on each sub-interval and  by using Proposition \ref{YA1}.  
The uniqueness of the solution of (P) follows from this apriori estimate. \\
Note that in Theorem \ref{thm2} we assume merely that $D(A(t_0)^{\frac{1}{2}}) = \V$ for some $t_0 \in [0, \tau]$. This  assumption implies the uniform Kato square root property by Proposition \ref{prop2.3}. 
\end{proof}

\begin{proof}[Proof of Proposition \ref{prop3}]
Set $B(t) := \mathcal{A}(t) - \mathcal{A}(0)$ and $E := H^1(0, \tau; \Hi) \cap L^\infty(0, \tau; \V)$. Then $E$ is a Banach space for the norm 
$$\| u \|_E := \| u \|_{H^1(0, \tau; \Hi)} + \| u \|_{L^\infty(0, \tau; \V)}.$$
Let $f \in L^2(0, \tau; \Hi)$ and $u_0 \in \V$. For $v \in E$, there exists a unique solution $u \in E$ to the problem 
\begin{equation*}
\left\{
  \begin{array}{rcl}
     u'(t) + A(0)\,u(t) &=& f(t) - B(t) v(t), \ t \in (0, \tau] \\
     u(0)&=&u_0.
  \end{array}
\right.
\end{equation*}
In addition, there exists a constant $C$, independent of $u_0, f$ and $v$ such that 
\begin{equation}\label{eqV1}
\| u \|_E \le C \left[ \| u_0 \|_{\V} + \|f \|_{L^2(0, \tau; \Hi)} + \| B v \|_{L^2(0, \tau; \Hi)} \right].
\end{equation}
Bounding $ \| u \|_{H^1(0, \tau; \Hi)}$ by the term on the RHS follows from the classical maximal regularity for the operator $A(0)$ in the Hilbert space $\Hi$. For the  bound of $\| u \|_{L^\infty(0, \tau; \V)}$ by the same term we use  either Proposition \ref{YA1} or the classical embedding 
$$H^1(0, \tau; \Hi) \cap L^2(0, \tau; D(A(0))) \hookrightarrow C([0, \tau]; D(A(0)^{\frac{1}{2}}).$$
Define the operator $K$ on $E$ by $K(v) := u$. We prove that for $\tau > 0$ small, $K$ is a contraction operator. Indeed, let  $v_1, v_2 \in E$. Then $w := K(v_1) - K(v_2)$ satisfies
\begin{equation*}
\left\{
  \begin{array}{rcl}
     w'(t) + A(0)\,w(t) &=&  - B(t) (v_1(t) - v_2(t)), \ t \in (0, \tau] \\
     u(0)&=&0.
  \end{array}
\right.
\end{equation*}
Hence by \eqref{eqV1} and [H4] with $\gamma = 0$
\begin{eqnarray*}
 \| w \|_E  &\le&  C \| B (v_1 - v_2) \|_{L^2(0, \tau; \Hi)} \\
 &=& C ( \int_0^\tau \| (A(t) - A(0))(v_1 - v_2)(t) \|^2 \, dt )^{\frac{1}{2}}\\
 &\le& C M_0 \sqrt{\tau} \| v_1 - v_2 \|_{L^\infty(0, \tau; \V)}\\
 &\le& C M_0 \sqrt{\tau} \| v_1 - v_2 \|_E.
 \end{eqnarray*}
 This shows that for $C M_0 \sqrt{\tau} < 1$ the operator $K$ is a contraction. Hence there exists a unique $u \in E$ such that 
 $K(u) = u$. This gives that $u$ satisfies (P) on $[0, \tau]$ for $\tau < \frac{1}{(CM_0)^2}$ and it follows from \eqref{eqV1} that $u$ satisfies the apriori estimate
 $$ \| u \|_E \le C'  [ \| u_0 \|_{\V} + \|f \|_{L^2(0, \tau; \Hi)} ].$$
 Finally, for arbitrary $\tau > 0$, we split $[0, \tau]$ into a finite number of  sub-intervals with small sizes and proceed exactly as in the previous proof.  
 \end{proof}

\section{Applications}\label{sec5}
In this section we give some applications of the previous results to concrete differential operators. 

\medskip
\noindent{\it -- Elliptic operators on $\R^n$.}    Let  $\Hi = L^{2}(\R^n)$  and $\V= H^{1}(\R^n)$  and define the sesquilinear forms 
$$\fra(t,u,v)= \sum_{k,l=1}^n \int_{\R^n} c_{kl} (t,x) \partial_k u \overline{\partial_l v} \, dx, \ \, u, v \in \V.$$
We assume that  the  matrix $C(t,x)  = (c_{kl}(t,x))_{1 \le k,l\le n}$ satisfies the usual ellipticity condition, that is,
 there
exists $\alpha, M >0$ such that 
$$ \alpha |\xi|^{2} \leq \Re(C(t,x)\xi.\bar{\xi})\  \mbox{and} \ |C(t,x)\xi.\nu| \leq M |\xi| |\nu|$$
for all $\xi,\nu \in \mathbb{C}^{n}$  and a.e $t \in [0, \tau]$, $x \in \R^n$.
 The forms $\fra(t)$  satisfy the assumptions [H1]-[H3]. For each $t$, the corresponding operator is formally given by
 $A(t) = -\sum_{k,l=1}^n \partial_l (c_{kl}(t,x) \partial_k)$. 
Next we assume  that  $C \in H^{\frac{1}{2}}(0,\tau; L^{\infty}(\C^{n^2}))$. 
We note that  $$\|\mathcal{A}(t)-\mathcal{A}(s)\|_{\mathcal{L}(\V, \V')} \le M' \|C(t,.)- C(s,.) \|_{L^{\infty}(\C^{n^2})}$$ for some contant $M'$. 
This implies that 
$\mathcal{A} \in H^{\frac{1}{2}}(0,\tau;\mathcal{L}(\V,\V'))$. 
We assume in addition that each $c_{kl}$ is H\"older continuous of order $\alpha $ for some $\alpha > 0$ with
\[
| c_{kl}(t,x) - c_{kl}(s,x) | \le c |t-s|^\alpha
\]
for a.e. $x \in \R^n$. This assumption implies in particular \eqref{eqHyp}. We could also weaken this assumption by formulating it in terms of the 
modulus of continuity, see  \eqref{eqmod}.\\
We are now allowed to  apply Theorem \ref{thm1}.  We  obtain maximal $L^2$-regularity and apriori estimate for the parabolic problem
\begin{equation*}
 \left\{
\begin{array}{l}
u'(t)+A(t) u(t)=f(t) \\
u(0)=u_{0}   \in H^1(\R^n).
\end{array}
\right.
\end{equation*}
That is,  for every $f \in L^{2}(0,\tau;L^{2}(\R^n))$ and $ u_{0} \in H^{1}(\R^n)$ there is unique  
solution $u \in H^{1}(0,\tau;L^{2}(\R^n))$.  Note that we also have from Proposition \ref{YA1} that  $u \in L^\infty(0, \tau ; H^1(\R^n))$. As we already mentioned before, the uniform Kato square root property required in Theorem \ref{thm1} is satisfied in this setting, see \cite{AHLMT}. 
As we mentioned in the introduction, maximal $L^2$-regularity for these elliptic operators was proved recently in \cite{AE} under the slightly stronger assumption that the coefficients satisfy  a BMO-$H^{\frac{1}{2}}$ regularity with respect to $t$. \\
The maximal $L^2$-regularity we proved here holds also in the case of elliptic operators on Lipschitz domains with Dirichlet or Neumann boundary conditions. The arguments are the same. One define the previous forms $\fra(t)$ with domain $\V = H^1_0(\Omega)$ (for Dirichlet boundary conditions) or $\V = H^1(\Omega)$ (for Neumann boundary conditions). 

\medskip
\noindent{-- \it Robin boundary conditions.} Let $\Omega$ be a bounded domain of $\R^d$ with Lipschitz boundary $\partial \Omega$. We denote by ${\rm Tr}$ the classical trace operator. Let  
$\beta: [0, \tau] \times \partial \Omega \to [0, \infty)$  be bounded and such that 
\begin{equation}\label{app2}
 \int_0^\tau \int_0^\tau \frac{\| \beta(t,.) - \beta(s,.) \|_{L^\infty(\partial \Omega)}^2}{| t-s|^{ 1 + 2\alpha} }\, ds dt < \infty
 \end{equation}
for some $\alpha> \frac{1}{4}$. In particular, $\beta (., x) \in H^\alpha$. We define the forms
$$\fra(t,u,v) := \int_{\Omega} \nabla u . \nabla v \, dx  + \int_{\partial \Omega} \beta(t,.) {\rm Tr} (u) {\rm Tr}(v)\, d \sigma,$$
 for all $ u, v \in \V := H^1(\Omega).$
Formally, the associated operator $A(t)$ is  (minus) the Laplacian with the  time dependent Robin boundary condition  
$$\tfrac{\partial u}{\partial n} + \beta(t,.) u = 0 \ {\rm on }\  \partial \Omega.$$
Here $\tfrac{\partial u}{\partial n}$ denotes the normal derivative in the weak sense.\\
Note that for any $\varepsilon > 0$ 
\begin{align*}
& | \fra(t;u,v) -\fra(s;u,v) | \\
&= | \int_{\partial \Omega} [\beta(t,.) - \beta(s,.)]  {\rm Tr}(u) {\rm Tr}(v)\, d \sigma | \\
&\le  \| \beta(t,.) - \beta(s,.) \|_{L^\infty(\partial \Omega)}  \| u \|_{H^{\frac{1}{2} + \varepsilon}(\Omega)} \| v\|_{H^{\frac{1}{2} + \varepsilon}(\Omega)},
\end{align*}
where we used the fat that the trace operator is bounded from $H^{\frac{1}{2} + \varepsilon}(\Omega)$ into $L_2(\partial \Omega)$.  
Now assumption \eqref{app2} allows us to apply Theorem \ref{thm2}  with $\gamma = \frac{1}{2} + \varepsilon$ and obtain maximal $L^2$-regularity for the corresponding evolution equation with initial data $u_0 \in H^1(\Omega)$. The forms considered here are symmetric and therefore the uniform Kato square root property can be checked easily. \\
Maximal $L^2$-regularity for time dependent Robin boundary condition with $\beta (., x) \in C^\alpha$ for some $\alpha > \frac{1}{4}$ was previously proved in \cite{AM} and \cite{Ou15}. In  \cite{Ou15}  maximal $L^p$-regularity is proved for all $p \in (1, \infty)$ is proved.

\medskip
\noindent{-- \it Operators with terms of lower order.} Let $\Omega$ be a domain of $\R^n$ and let  $b_k, m :[0, \tau] \times \Omega \to \R$ be a bounded measurable function for each $k= 1, \cdots, n$. We define the forms
$$\fra(t,u,v) = \int_\Omega \nabla u . \nabla v dx + \sum_{k=1}^n \int_\Omega b_k(t,x) \partial_k u v dx + \int_\Omega m(t,x) u v dx,$$
with domain $\V$, a closed subset of $H^1(\Omega)$ which contains $H^1_0(\Omega)$. 
It is clear that 
$$| \fra(t,u,v) - \fra(s,u,v) | \le M_0 \| u \|_{\V} \| v\|_2$$
for some constant $M_0$. 
This means that assumption [H4] is satisfied with $\gamma = 0$. We apply Proposition \ref{prop3} and obtain maximal $L^2$-regularity for the correspond evolution equation.\\
 As we mentioned in the introduction, the domains of the corresponding operators are independent of $t$ and one may apply the results from \cite{ACFP} to obtain maximal regularity. Doing so, one needs  to assume some regularity  with respect to $t$ for the coefficients $b_k(t,x)$ and $m(t,x)$ whereas the result we obtain from Proposition \ref{prop3} does not require any regularity. \\
 
 \noindent{\bf Acknowledgements.} The authors wish to thank Stephan Fackler for a useful discussion on the content  of this paper.

\end{document}